\documentclass[a4paper,10pt]{article}

\usepackage{amsmath,amstext,amsthm,amsfonts,xspace}

\usepackage[%
   pdfpagemode=UseNone,
   bookmarks=true,
   colorlinks,
   linkcolor=blue,
   anchorcolor=blue,
   citecolor=blue,
   filecolor=blue,
   pagecolor=blue,
   urlcolor=blue
   ]{hyperref}

\usepackage{amssymb}
\usepackage{comment}
\usepackage{graphicx}
\usepackage{psfrag}
\usepackage{verbatim}
\usepackage{fullpage}

\theoremstyle{plain}
   \newtheorem{theorem}{Theorem}
   \newtheorem{corollary}{Corollary}
   \newtheorem{lemma}{Lemma}
   
\theoremstyle{definition}
   \newtheorem{definition}{Definition}
   
   \newtheorem{assumptions}{Assumption}
   \newtheorem{remark}{Remark}
\theoremstyle{remark}

\numberwithin{theorem}{section}
\numberwithin{equation}{section}

\newcommand{\note}[1]%
{\noindent\centerline{\fbox{\parbox{.9\textwidth}{\textbf{#1}}}}}
\newcommand{\snote}[1]%
{\fbox{\textbf{#1}}}

\newcommand{\step}[1]{\noindent\raisebox{1.5pt}[10pt][0pt]{\tiny\framebox{$#1$}}\xspace}

\newcommand{\definedas}{:=}

\newcommand{\D}{\displaystyle}	

\newcommand{\normaa}[1]{\left\| {#1} \right\|_a}
\newcommand{\normab}[1]{\left\| {#1} \right\|_b}
\newcommand{\norm}[1]{\left\| {#1} \right\|_{\Omega}}
\newcommand{\normT}[1]{\left\| {#1} \right\|_{T}}
\newcommand{\normNT}[1]{\left\| {#1} \right\|_{\omega_\Tau(T)}}
\newcommand{\normNS}[1]{\left\| {#1} \right\|_{\omega_\Tau(S)}}
\newcommand{\normS}[1]{\left\| {#1} \right\|_{S}}
\newcommand{\normbT}[1]{\left\| {#1} \right\|_{\partial T}}
\newcommand{\normkNT}[1]{\left\| {#1} \right\|_{\omega_k(T)}}

\DeclareMathOperator{\dist}{dist}
\DeclareMathOperator{\Span}{span}
\DeclareMathOperator{\diam}{diam}

\newcommand{\est}[2]{\eta_\Tau(#1;#2)}
\newcommand{\gest}[1]{\eta_\Tau(#1)}

\newcommand\NN{\mathbb N}        

\newcommand\AAA{\mathcal A}
\newcommand\BB{\mathcal B}
\newcommand\MM{\mathcal M}
\newcommand\Tau{\mathcal T}
\newcommand\PP{\mathcal P}
\newcommand\VV{\mathbb V}
\newcommand\WW{\mathbb W}
\newcommand\OO{\mathcal O}
\newcommand\RRR{\mathcal R}
\newcommand\SSS{\mathcal S}

\newcommand{\RR}{\mathbb R}         

\newcommand{\NT}{{\mathcal N}_\Tau}
\newcommand{\NTk}{{\mathcal N}_{\Tau_k}}
\newcommand{\Nk}{{\mathcal N}_k}

\title{Convergence of adaptive finite element methods \\for eigenvalue problems}
\author{%
Eduardo M.\ Garau$^*$ 
\and 
Pedro Morin\thanks{Partially supported by CONICET (Argentina) through Grant PIP 5478, and Universidad Nacional del Litoral through Grant CAI+D 12/H421} 
\and 
Carlos Zuppa\thanks{Partially supported by Universidad Nacional de San Luis through Grant 22/F730-FCFMyN.}}

\date{}

\begin{document}

\maketitle

\begin{abstract}
In this article we prove convergence of adaptive finite element methods for second order elliptic eigenvalue problems. We consider Lagrange finite elements of any degree and prove convergence for simple as well as multiple eigenvalues under a minimal refinement of marked elements, for all \emph{reasonable} marking strategies, and starting from any initial triangulation.
\end{abstract}

\section{Introduction and main result}\label{S:intro}

In many practical applications it is of interest to find or approximate the eigenvalues and eigenfunctions of elliptic problems. Finite element approximations for these problems have been widely used and analyzed under a general framework. Optimal a priori error estimates for the eigenvalues and eigenfunctions have been obtained (see~\cite{Babuska_eigenvalue, Babuska, Raviart-Thomas, Strang-Fix} and the references therein).

Adaptive finite element methods are an effective tool for making an efficient use of the computational resources; for certain problems, it is even indispensable to their numerical resolvability.  A quite popular, natural adaptive version of classical finite element methods consists of the loop
\begin{equation*}
%
 \textsc{Solve $\to$ Estimate $\to$ Mark $\to$ Refine},
\end{equation*}
that is: solve for the finite element solution on the current grid, compute the a~posteriori error estimator, mark with its help elements to be subdivided, and refine the current grid into a new, finer one. The ultimate goal of adaptive methods is to equidistribute the error and the computational effort obtaining a sequence of meshes with optimal complexity. Historically, the first step to prove optimality has always been to understand convergence of adaptive methods. A general result of convergence for linear problems has been obtained in~\cite{MSV_convergence}, where very general conditions on the linear problems and the adaptive methods that guarantee convergence are stated. Optimality for adaptive methods using D\"orfler's~\cite{Dorfler} marking strategy has been proved in~\cite{CKNS-quasi-opt, Stevenson} for linear problems.

The goal of this article is to analyze the convergence of adaptive finite element methods for the eigenvalue problem consisting in finding $\lambda \in \RR$, and $u\not\equiv 0$ such that
\[
 -\nabla\cdot(\AAA \nabla u) = \lambda \BB u \quad \text{in $\Omega$},
\qquad 
u=0\quad\text{on $\partial\Omega$},
\]
under general assumptions on $\AAA$, $\BB$ and $\Omega$ that we state precisely in Section~\ref{S:setting}.

As we mentioned before, adaptive methods are based on a posteriori error estimators, that are computable quantities depending on the discrete solution and data, and indicate a distribution of the error. A posteriori error estimators for eigenvalue problems have been constructed by using different approaches in~\cite{Verfurth-book, Verfurth-nonlinear, Duran, Larson}, they have been developed for $\AAA \equiv I$ and $\BB \equiv 1$, but the same proofs can be carried over to the general case considered here; see~\cite{Giani} and Section~\ref{S:apost}. An important aspect to be mentioned here is that the upper bound holds for sufficiently fine meshes. However, our proof will not rely on this bound, allowing us to prove convergence from any initial mesh. The first result (and only up to now) about convergence of adaptive finite elements for eigenvalue problems has been presented in~\cite{Giani}.

The following is the main result of this article.

\begin{theorem}[Main Result] Let $\lambda_k$ and $u_k$ be the discrete eigenvalues and eigenfunctions obtained with the adaptive algorithm stated in Section~\ref{S:adloop} below. Then there exists an eigenvalue $\lambda$ of the continuous problem such that
\[
 \lim_{k\rightarrow\infty} \lambda_k=\lambda \qquad \textrm{and}\qquad \lim_{k\rightarrow\infty} \dist_{H^1_0(\Omega)}(u_k,M(\lambda)) = 0,
\]
where $M(\lambda)$ denotes the set of all eigenfunctions of the continuous problem corresponding to the eigenvalue $\lambda$.
\end{theorem}

\begin{remark}
 Before proceeding with the details of the statement and the proof of this result, we make some remarks:
\begin{itemize}
\item An important difference with previous works is that we do not require the initial mesh $\Tau_0$ to be fine enough. Any initial mesh that captures the discontinuities of $\AAA$ will guarantee convergence.
\item The result holds for any of the popular marking strategies, not only D\"orfler's~\cite{Dorfler}. The only assumption is that non-marked elements have error estimators smaller than marked ones, see condition~(\ref{E:marking}) in Section~\ref{S:adloop} below.
\item The marking is done according to the residual type a posteriori error estimators presented in Section~\ref{S:apost}. Even though there are some \emph{oscillation terms} in the lower bound, we do not require any marking due to these terms. We only need to mark according to the error estimators, which is what is usually done in practice.
\item The result holds with a minimal refinement of marked elements, one bisection suffices. We do not require the enforcement of the so-called \emph{interior node property}.
\item We are assuming that each of the discrete eigenvalues $\lambda_k$ is the $j$-th eigenvalue of the corresponding discrete problem. The result, as stated above, only guarantees that $\lambda_k$ converges to one eigenvalue $\lambda$ of the continuous problem. We can be sure that we approximate the $j$-th eigenvalue of the continuous problem under any of the following assumptions:
  \begin{itemize}
  \item No eigenfunction is equal to a polynomial of degree $\le \ell$ on an open region of $\Omega$, where $\ell$ denotes the polynomial degree of the finite element functions being used. This is a Non-Degeneracy Assumption, and it holds for a large class of problems; see Assumption~\ref{A:non-deg} and following discussion.
  \item The meshsize of the initial triangulation is small enough. This assumption goes against the spirit of adaptivity and a posteriori analysis, since we cannot quantify what \emph{small enough} means. But we state it for completeness, because in some (nonlinear) problems there may be no way to overcome this.
  \end{itemize}
\item The proof follows similar ideas to those of~\cite{MSV_convergence}, with some modifications due to the different nature of the problem. It consists in proving the following steps:
  \begin{itemize}
  \item The full sequence of discrete eigenvalues converges to a number $\lambda_\infty$ and a subsequence of the discrete eigenfunctions converges to some function $u_\infty$.
  \item The global a posteriori error estimator converges to zero (for the subsequence).
  \item The pair $(\lambda_\infty, u_\infty)$ is an eigenpair of the continuous problem. Due to a lack of a sharp upper bound (it only holds for sufficiently fine meshes) it is necessary to introduce a new argument to prove this (see Theorem~\ref{T:eigenfunction}). This new argument is perhaps the main difference with respect to~\cite{MSV_convergence}, and we believe that the idea can be useful for many other nonlinear problems.
  \end{itemize}
\end{itemize}

\end{remark}

The rest of the article is organized as follows. In Section~\ref{S:problem} we state precisely the problem that we study, describe the approximants, mention some already known results about a priori and a posteriori estimation, and state the adaptive loop. In Section~\ref{S:uinfty} we prove that the sequence $\{(\lambda_k,u_k)\}_{k\in\NN_0}$ of solutions to the discrete problems contains a subsequence that converges to a limiting pair $(\lambda_\infty,u_\infty)$. In Section~\ref{S:convest} we prove that the global a posteriori error estimator tends to zero; which is instrumental to conclude in Section~\ref{S:eigenfunction} that $(\lambda_\infty,u_\infty)$ is an eigenpair of the continuous problem. Finally, in Section~\ref{S:main-result} we state and prove the main result and discuss its implications.

\section{Problem statement and adaptive algorithm}\label{S:problem}

This section is subdivided in four parts. In Section~\ref{S:setting} we state precisely the continuous problem that we study and mention some of its properties. In Section~\ref{S:discrete problem} we state the discrete problems that we consider as approximants to the continuous one, mention some of its properties and state the a priori error estimates. In Section~\ref{S:apost} we define the a posteriori error estimators that we use, state the upper bound and prove the discrete local lower bound that we will use in our convergence proof. Finally, in Section~\ref{S:adloop} we state the adaptive algorithm together with the assumptions on each of its blocks.

\subsection{Setting}\label{S:setting}
Let $\Omega\subset \RR^d$ be a bounded open set with a Lipschitz boundary. In particular, we suppose that $\Omega$ is a polygonal domain if $d=2$ and a polyhedral domain if $d=3$. Let $a,b: H^1_0(\Omega) \times H^1_0(\Omega) \to \RR$ be the bilinear forms defined by
\begin{equation*}
a(u,v)\definedas\int_\Omega \AAA \nabla u\cdot\nabla v,
\end{equation*}
and
\begin{equation*}
b(u,v)\definedas\int_\Omega \BB u v,
\end{equation*}
where $\AAA$ is a piecewise $W^{1,\infty}(\Omega)$ symmetric-matrix-valued function which is uniformly positive definite, i.e., there exist constants $a_1,a_2>0$ such that
\begin{equation*}
a_1|\xi|^2\leq \AAA(x)\xi\cdot \xi\leq a_2|\xi|^2,\qquad \forall~\xi\in\RR^d,\qquad \forall~x\in\Omega,
\end{equation*}
and $\BB$ is a scalar function such that
$$b_1\leq \BB(x)\leq b_2,\qquad \forall~x\in\Omega,$$
for some constants $b_1,b_2>0$.

We also define the induced norms by these bilinear forms as
\begin{equation*}
\normaa{v}\definedas a(v,v)^{1/2}, \quad v\in H_0^1(\Omega), \qquad\text{and}\qquad
\normab{v}\definedas b(v,v)^{1/2}, \quad v\in L^2(\Omega).
\end{equation*}
By the assumptions on $\AAA$ and $\BB$,  $\normaa{\cdot}\simeq\|\cdot\|_{H^1_0(\Omega)}$ and $\normab{\cdot}\simeq\|\cdot\|_{\Omega}$, i.e., there exist positive constants $c_1$, $c_2$, $c_3$, $c_4$ such that
$$c_1 \|v\|_{H^1_0(\Omega)}\leq \normaa{v}\leq c_2 \|v\|_{H^1_0(\Omega)},\qquad\forall~v\in H^1_0(\Omega),$$
and
$$c_3 \|v\|_{\Omega}\leq \normab{v}\leq c_4 \|v\|_{\Omega},\qquad\forall~v\in L^2(\Omega).$$
Where, hereafter, if $A\subset\Omega$, $\|\cdot\|_A$ will denote the $L^2(A)$-norm.

We consider the following 
\paragraph{Continuous eigenvalue problem.} Find $\lambda\in\RR$ and $u\in H^1_0(\Omega)$ satisfying
\begin{equation} \label{E:cont-problem}
\left\{
\begin{array}{l}
a(u,v)=\lambda \, b(u,v),\qquad \forall~v\in  H^1_0(\Omega),\\
\normab{u}=1.
\end{array}
\right.
\end{equation}

It is well known~\cite{Babuska_eigenvalue} that under our assumptions on $\AAA$ and $\BB$ problem~(\ref{E:cont-problem}) has a countable sequence of eigenvalues 
$$0<\lambda_1\leq \lambda_2\leq \lambda_3\leq \ldots \nearrow \infty$$
and corresponding eigenfunctions
$$u_1,u_2,u_3,\ldots$$
which can be assumed to satisfy
$$b(u_i,u_j)=\delta_{ij}
:=\begin{cases} 1 &i=j,\\ 0 &i\neq j,\end{cases}$$
where in the sequence $\{\lambda_j\}_{j\in\NN}$, the $\lambda_j$ are repeated according to geometric multiplicity.

Also, the eigenvalues can be characterized as extrema of the Rayleigh quotient
$\D\RRR(u)=\frac{a(u,u)}{b(u,u)},$
by the following relationships.
\begin{itemize}
 \item \textbf{Minimum principle:}
\begin{align*}
\lambda_1&=\min\limits_{u\in H^1_0(\Omega)}\RRR(u)=\RRR(u_1), \\
\lambda_j&=\min\limits_{\substack{u\in H^1_0(\Omega)\\a(u,u_i)=0\\i=1,\ldots,j-1}}\RRR(u)=\RRR(u_j),\quad j=2,3,\dots.
\end{align*}

\item \textbf{Minimum-maximum principle:}
\begin{equation*}
\lambda_j=\min\limits_{\substack{V_j\subset H^1_0(\Omega)\\ \dim V_j=j}}\max\limits_{u\in V_j}\RRR(u)=\max\limits_{u\in \Span\{u_1,\ldots,u_j\}}\RRR(u),\quad j=1,2,\ldots.
\end{equation*}
\end{itemize}

For each fixed eigenvalue $\lambda$ of (\ref{E:cont-problem}) we define 
\begin{equation*}
M(\lambda)\definedas\{u\in H^1_0(\Omega):~u~\textrm{satisfies (\ref{E:cont-problem})}\},
\end{equation*}
and notice that if $\lambda$ is simple, then $M(\lambda)$ contains two functions, whereas if $\lambda$ is not simple, it consists of a sphere in the subspace generated by the eigenfunctions.

\subsection{Discrete problem}\label{S:discrete problem}

In order to define the discrete approximations we will consider conforming triangulations $\Tau$ of the domain $\Omega$, that is, partitions of $\Omega$ into $d$-simplices such that if two elements intersect, they do so at a full vertex/edge/face of both elements. For any triangulation $\Tau$, $\SSS$ will denote the set of interior sides, where by side we mean an edge if $d=2$ and a face if $d=3$. And $\kappa_\Tau$ will denote the regularity of $\Tau$, defined as
$$\kappa_\Tau\definedas \max_{T\in\Tau} \frac{\diam(T)}{\rho_T},$$
where $\diam(T)$ is the length of the longest edge of $T$, and  $\rho_T$ is the radius of the largest ball contained in it. It is also useful to define the meshsize $h_\Tau \definedas \D\max_{T\in\Tau} h_T$, where $h_T\definedas|T|^{1/d}$.

Let $\ell\in\NN$ be fixed, and let $\VV_\Tau$ be the finite element space consisting of continuous functions vanishing on $\partial \Omega$ which are polynomials of degree $\le \ell$ in each element of $\Tau$, i.e,
\begin{equation*}
\VV_\Tau\definedas\{v\in H^1_0(\Omega):\quad v|_T\in \PP_\ell(T),\quad \forall~T\in\Tau\}.
\end{equation*}
Obviously, $\VV_\Tau\subset H^1_0(\Omega)$ and if $\Tau_*$ is a refinement of $\Tau$, then $\VV_\Tau\subset\VV_{\Tau_*}$.

We consider the approximation of the continuous eigenvalue problem~(\ref{E:cont-problem}) with the following

\paragraph{Discrete eigenvalue problem.} Find $\lambda_\Tau\in \RR$ and $u_\Tau\in\VV_\Tau$ such that
\begin{equation}\label{E:disc-problem}
\left\{
\begin{array}{l}
a(u_\Tau,v)=\lambda_\Tau \, b(u_\Tau,v),\qquad \forall~v\in  \VV_\Tau,\\
\normab{u_\Tau}=1.
\end{array}
\right.
\end{equation}

For this discrete problem, similar results to those of the continuous problem hold~\cite{Babuska_eigenvalue}. More precisely, problem~(\ref{E:disc-problem}) has a finite sequence of eigenvalues 
$$0<\lambda_{1,\Tau}\leq \lambda_{2,\Tau}\leq \lambda_{3,\Tau}\leq \ldots \leq \lambda_{N_\Tau,\Tau},$$
where $N_\Tau\definedas \dim \VV_\Tau$, and corresponding eigenfunctions
$$u_{1,\Tau},u_{2,\Tau},u_{3,\Tau},\ldots,u_{N_\Tau,\Tau},$$
which can be assumed to satisfy
$$b(u_{i,\Tau},u_{j,\Tau})=\delta_{ij}.$$

Moreover, the following extremal characterizations also hold:
\begin{itemize}
 \item \textbf{Minimum principle:}
\begin{align*}
\lambda_{1,\Tau}&=\min\limits_{u\in \VV_\Tau}\RRR(u)=\RRR(u_{1,\Tau}),\\
\lambda_{j,\Tau}&=\min\limits_{\substack{u\in \VV_\Tau\\a(u,u_{i,\Tau})=0\\i=1,\ldots,j-1}}\RRR(u)=\RRR(u_{j,\Tau}),\quad j=2,3,\ldots,N_\Tau.
\end{align*}

\item \textbf{Minimum-maximum principle:}
\begin{equation*}
\lambda_{j,\Tau}=\min\limits_{\substack{V_{j,\Tau}\subset \VV_\Tau\\ \dim V_{j,\Tau}=j}}\max\limits_{u\in V_{j,\Tau}}\RRR(u)=\max\limits_{u\in \Span\{u_{1,\Tau},\ldots,u_{j,\Tau}\}}\RRR(u),\quad j=1,2,\ldots,N_\Tau.
\end{equation*}
\end{itemize}

It follows from the minimum-maximum principles that 
$$\lambda_j\leq \lambda_{j,\Tau},\qquad j=1,2,\ldots,N_\Tau.$$
and it also follows that if $\Tau_*$ is any refinement of $\Tau$ then
$$\lambda_{j,\Tau_*}\leq \lambda_{j,\Tau},\qquad j=1,2,\ldots,N_\Tau.$$

For a given eigenvalue $\lambda$ we define a notion of minimal error of approximation of its eigenfunctions by
\begin{equation*}
\epsilon_\Tau(\lambda)\definedas\sup_{u\in M(\lambda)} \inf_{\chi\in\VV_\Tau} \normaa{u-\chi}.
\end{equation*}
For $j=1,2,\ldots,N_\Tau$, there holds that
$$\lambda_{j,\Tau}-\lambda_j \lesssim \epsilon_\Tau^2(\lambda_j),$$
where, from now on, whenever we write $A \lesssim B$ we mean that $A \le C B$ with a constant $C$ that may depend on $\AAA$, $\BB$, the domain $\Omega$ and the regularity $\kappa_\Tau$ of $\Tau$, but not on other properties of $\Tau$ such as element size or uniformity.
	
If $\{\Tau_k\}_{k\in\NN_0}$ is any sequence of triangulations such that $\D\sup_{k\in\NN_0} \kappa_{\Tau_k} < \infty$, and $h_{\Tau_k} \to 0$ as $k\to\infty$, then
$$\epsilon_{\Tau_k}(\lambda_j)\longrightarrow 0,\qquad\textrm{as}\quad k \longrightarrow \infty,$$
and therefore,
\begin{equation}\label{E:lambdak converges}
\lambda_{j,\Tau_k}\longrightarrow\lambda_j,\qquad\textrm{as}\quad k\longrightarrow \infty	.
\end{equation}
This holds for any $j\in\NN$ and it is a consequence of standard interpolation estimates and the fact that $M(\lambda_j)$ is bounded and contained in a finite dimensional subspace of $H^1_0(\Omega)$.

\subsection{A posteriori error estimators}\label{S:apost}

 A posteriori estimates for eigenvalue problems have been studied by Larson~\cite{Larson}, Dur\'an, Padra and Rodr\'iguez~\cite{Duran}, Giani and Graham~\cite{Giani}. In this section we present the residual type a posteriori estimates for eigenvalue problems, state without proof some already known properties and prove the discrete local lower bound that will be useful for our convergence proof.

In order to define the estimators we assume that the triangulation $\Tau$ matches the discontinuities of $\AAA$. More precisely, we assume that the discontinuities of $\AAA$ are aligned with the sides of $\Tau$. Observe that in particular, $\AAA|_T$ is Lipschitz continuous for all $T\in\Tau$.

\begin{definition}[Element residual and jump residual]
For $\mu\in\RR$ and $v\in \VV_\Tau$ we define the element residual $R(\mu,v)$ by
\begin{equation}\label{E:element-residual}
R(\mu,v)|_{T}\definedas -\nabla \cdot (\AAA\nabla v)-\mu \BB v,
\end{equation}
for all $T\in\Tau$, and the jump residual $J(v)$ by
\begin{equation}\label{E:jump-residual}
J(v)|_{S}\definedas (\AAA\nabla v)|_{T_1}\cdot \overrightarrow{n_1}+(\AAA\nabla v)|_{T_2}\cdot \overrightarrow{n_2},
\end{equation}
for every interior side $S\in \SSS$, where $T_1$ and $T_2$ are the elements in $\Tau$ which share $S$ and  $\overrightarrow{n_i}$ is the outward normal unit vector of $T_i$ on $S$, for $i=1,2$. We define $J(v)|_{\partial \Omega}\definedas 0$.
\end{definition}

\begin{definition}[Local and global error estimator]
For $\mu\in\RR$ and $v\in \VV_\Tau$ we define the local error estimator $\est{\mu,v}{T}$ by
\begin{equation*}
\est{\mu,v}{T}^2\definedas h_T^2\normT{R(\mu,v)}^2 + h_T\normbT{J(v)}^2,
\end{equation*}
for all $T\in\Tau$, and the global error estimator $\gest{\mu,v}$ is given by
\begin{equation*}
\gest{\mu,v}^2\definedas \sum_{T\in\Tau} \est{\mu,v}{T}^2.
\end{equation*}
\end{definition}

Even though we will not need it for the convergence proof, we include the statement of the upper bound of the error in terms of the a posteriori error estimation, for the sake of completeness.

\begin{theorem}[Upper bound]\label{T:upperbound}
Let $j\in\NN$, and let $u_{\Tau}$ be an eigenfunction corresponding to the $j$-th eigenvalue $\lambda_{\Tau}$ of the discrete problem~\eqref{E:disc-problem}, then, if $h_\Tau$ is small enough, there exists an eigenfunction $u$ corresponding to the $j$-th eigenvalue $\lambda$ of the continuous problem~\eqref{E:cont-problem} such that 
$$\normaa{u-u_\Tau} \lesssim \gest{\lambda_\Tau,u_\Tau}.$$
\end{theorem}

The proof of this theorem can be obtained following the steps given in~\cite{Duran}, by extending Lemmas 3.1 and 3.2 presented there for the model problem with $\AAA \equiv I$, and $\BB \equiv 1$, to the general case that we consider here, using the following regularity result, and the a priori bound stated in Theorem~\ref{T:apriori} below.

\begin{lemma}[Regularity of the eigenfunctions]\label{L:regularity}
There exists $r\in (0,1]$ depending only on $\Omega $ and $\AAA$ such that
\begin{equation*}
u\in H^{1+r}(\Omega ),
\end{equation*}%
for any eigenfunction $u$ of the problem~(\ref{E:cont-problem}).
\end{lemma}
\begin{proof}
This can be proved by observing that if $u$ is an eigenfunction, then it is also a solution to a linear elliptic equation of second order with right-hand side in $L^2(\Omega)$. We know that $r=1$ when $\AAA$ is constant or smooth and $\Omega$ is convex. The case in which $\Omega$ is non-convex has been studied in~\cite{BDLN} and the case of $\AAA$ having a discontinuity across an interior interface in~\cite{Babuska70}. For the general case, which we are considering here, see~\cite[Theorem 3]{Jochmann}.
\end{proof}

The following result is an a priori estimate relating the errors in the strong and weak norms associated to the problem, and it is the last slab in the chain necessary to prove Theorem~\ref{T:upperbound}. The case $\AAA \equiv I$ and $\BB = 1$ can be easily obtained from the results in~\cite{Strang-Fix}, and in~\cite{Raviart-Thomas}. The general case was presented in~\cite{Giani,Giani-thesis}.

\begin{theorem}\label{T:apriori}
Let the same assumptions of Theorem~\ref{T:upperbound} hold. Then, if $h_\Tau$ is small enough, there exists an eigenfunction $u$ corresponding to the $j$-th eigenvalue $\lambda$ of the continuous problem~\eqref{E:cont-problem} such that 
\begin{equation*}
\normab{u-u_\Tau}\lesssim h_\Tau^r\normaa{u-u_\Tau}.
\end{equation*}
\end{theorem}

The next result, which we will need for our proof of convergence is the discrete local lower bound, whose proof follows that of the continuous lower bound in~\cite{Duran}, but in order to make this article more self-contained we will include it here.

For $S\in \SSS$ we define $\omega_\Tau (S)$ as the union of the two elements in $\Tau$ sharing $S$. For $T\in\Tau$,  $\NT(T)\definedas\{T'\in\Tau : T'\cap T \neq \emptyset\}$ denotes the set of neighbors of $T$ in $\Tau$, and  $\omega_\Tau(T) \definedas \bigcup_{T'\in\NT(T)}T'$.
We also define $n_d\definedas 3$ if $d=2$ and $n_d\definedas 6$ if $d=3$. This guarantees that after $n_d$ bisections to an element, new nodes appear on each side and in the interior. Here we consider the newest-vertex bisection in two dimensions and the procedure of Kossaczk\'y in three dimensions~\cite{Alberta}.

\begin{theorem}[Discrete local lower bound]\label{T:general_lower_bound}
Let $T\in\Tau$ and let $\Tau_*$ be the triangulation of $\Omega$ which is obtained from $\Tau$ by bisecting $n_d$ times each element of $\NT(T)$. Let $\lambda_\Tau$ and $u_\Tau$ be a solution to the discrete problem~(\ref{E:disc-problem}). Let $\WW$ be a subspace of $H^1_0(\Omega)$ such that $\VV_{\Tau_*}\subset \WW$. If $\mu\in\RR$ and $w\in\WW$ satisfy   
\begin{equation*}
\left\{
\begin{array}{l}
a(w,v)=\mu \, b(w,v),\qquad \forall~v\in  \WW,\\
\normab{w}=1,
\end{array}
\right.
\end{equation*}
then 
$$\est{\lambda_\Tau,u_\Tau}{T}\lesssim  \normNT{\nabla (w-u_\Tau)} +h_T \normNT{\mu w-\lambda_\Tau u_\Tau} +h_T \normNT{R-\overline{R}}+h_T^{1/2} \normbT{J-\overline{J}},$$
where, for every $T'\in \NT(T)$,  $\overline{R}|_{T'}$ is the $L^2(T')$-projection of $R\definedas R(\lambda_\Tau,u_\Tau)$ onto $\PP_{\ell -1}$, and for every side $S\subset\partial T$, $\overline{J}|_{S}$ is the $L^2(S)$-projection of $J\definedas J(u_\Tau)$ onto $\PP_{\ell -1}$.
\end{theorem}

\begin{proof}
\step{1} We first analyze the element residual. We obviously have
\begin{equation}\label{E:lowbound_aux1}
\normT{R}\leq \normT{\overline{R}}+\normT{R-\overline{R}}.
\end{equation}
Let $x_T^{int}$ denote the vertex of $\Tau_*$ which is interior to $T$. Let $\varphi_T$ be the continuous piecewise linear function over $\Tau_*$ such that $\varphi_T(x_T^{int})=1$ and $\varphi_T$ vanishes over all the others vertices of $\Tau_*$. Then
\begin{equation}\label{E:interior-residual-split}
 \normT{\overline{R}}^2 \lesssim  \int_T \overline{R}^2 \varphi_T
= \int_T \overline{R}(\overline{R} \varphi_T)
= \int_T R(\overline{R} \varphi_T)+\int_T (\overline{R}-R)\overline{R} \varphi_T.
\end{equation}

If we define $v\definedas\overline{R} \varphi_T\in\VV_{\Tau_*}\subset \WW$, taking into account that $v$ vanishes over $\partial T$, for the first integral in~(\ref{E:interior-residual-split}) we have
\begin{align*}
\int_T Rv &= \int_T (-\nabla \cdot (\AAA\nabla u_\Tau)-\lambda_\Tau \BB u_\Tau)v\\
&= \int_T \AAA\nabla u_\Tau \cdot\nabla v -\int_T \lambda_\Tau \BB u_\Tau v\\
&= \int_T \AAA\nabla u_\Tau \cdot\nabla v -\int_T \lambda_\Tau \BB u_\Tau v-\int_T \AAA\nabla w \cdot \nabla v+\int_T \mu \BB w v\\
&= \int_T \AAA\nabla (u_\Tau-w) \cdot\nabla v +\int_T \BB(\mu w-\lambda_\Tau u_\Tau)v\\
&\lesssim \normT{\nabla (u_\Tau-w)}\normT{\nabla v}+\normT{\mu w-\lambda_\Tau u_\Tau}\normT{v} .
\end{align*}

For the second integral in~(\ref{E:interior-residual-split}) we have
\begin{align*}
\int_T (\overline{R}-R) \overline{R}\varphi_T \leq \normT{\overline{R}\varphi_T}\normT{\overline{R}-R}
\leq \normT{\overline{R}}\normT{\overline{R}-R}.
\end{align*}

Therefore, taking into account that $\normT{\nabla v}\lesssim \frac{1}{h_T}\normT{v}$ and $\normT{v}\leq \normT{\overline{R}}$ we can write
\begin{equation*}
\normT{\overline{R}}^2 \lesssim \normT{\nabla (u_\Tau-w)}\frac{1}{h_T}\normT{\overline{R}}+\normT{\mu w-\lambda_\Tau u_\Tau}\normT{\overline{R}} +\normT{\overline{R}}\normT{\overline{R}-R},
\end{equation*}
and then
\begin{equation}\label{E:lowbound_aux2}
h_T\normT{\overline{R}}\lesssim \normT{\nabla (u_\Tau-w)}+h_T\normT{\mu w-\lambda_\Tau u_\Tau} +h_T\normT{\overline{R}-R}.
\end{equation}
Now, from~(\ref{E:lowbound_aux1}) and~(\ref{E:lowbound_aux2}) it follows that
\begin{equation}\label{E:lowbound-int}
h_T\normT{R}\lesssim \normT{\nabla (u_\Tau-w)}+h_T\normT{\mu w-\lambda_\Tau u_\Tau}+h_T\normT{R-\overline{R}} .
\end{equation}
The same bound holds replacing $T$ by $T'$, for all $T'\in \NT(T)$.

\step{2} Secondly, we estimate the jump residual. Let $S$ be a side of $T$ and let $T_1$ and $T_2$ denote the elements sharing $S$. Obviously, one of them is $T$ itself. As before we proceed by bounding first the projection $\overline{J}$ of $J$, since
\begin{equation}\label{E:lowbound_aux3}
\normS{J}\leq \normS{\overline{J}}+\normS{J-\overline{J}}.
\end{equation}
Let $x_S^{int}$ denote the vertex of $\Tau_*$ which is interior to $S$. Let $\varphi_S$ be the continuous piecewise linear function over $\Tau_*$ such that $\varphi_S(x_S^{int})=1$ and $\varphi_S$ vanishes over all the others vertices of $\Tau_*$. Then

\begin{equation}\label{E:split-jump}
\normS{\overline{J}}^2 \lesssim \int_S (\overline{J})^2 \varphi_S
= \int_S \overline{J}(\overline{J} \varphi_S)
= \int_S J(\overline{J} \varphi_S)+\int_S (\overline{J}-J)\overline{J} \varphi_S.
\end{equation}

Now, we extend $\overline{J}$ to $\omega_\Tau(S)$ as constant along the direction of one side of each $T_i$, for $i=1,2$, and still call this extention $\overline{J}$. Observe that $\overline{J}$ is continuous on $\omega_\Tau(S)$ and $\overline{J}|_{T_i}\in \PP_{\ell-1}(T_i)$, for $i=1,2$.

Since $v\definedas\overline{J} \varphi_S\in\VV_{\Tau_*}\subset \WW$ and taking into account that $v=0$ on $\partial (\omega_\Tau (S))$, for the first integral in~(\ref{E:split-jump}) we have
\begin{align*}
\int_S Jv &= \sum_{i=1,2} \int_{\partial T_i}v\AAA\nabla u_\Tau\cdot \overrightarrow{n_i}=\sum_{i=1,2} \int_{T_i}\nabla\cdot(v\AAA\nabla u_\Tau)\\
&= \sum_{i=1,2} \left(\int_{T_i} \AAA\nabla u_\Tau\cdot \nabla v + \int_{T_i} v\nabla\cdot(\AAA\nabla u_\Tau)\right)\\
&= \sum_{i=1,2} \left(\int_{T_i} \AAA\nabla u_\Tau\cdot \nabla v + \int_{T_i} v\nabla\cdot(\AAA\nabla u_\Tau)\right)+\int_{T_1\cup T_2} \mu \BB w v-\int_{T_1\cup T_2} \AAA\nabla w\cdot \nabla v\\
&= \sum_{i=1,2} \left(\int_{T_i} \AAA\nabla (u_\Tau-w)\cdot \nabla v + \int_{T_i} (\nabla\cdot(\AAA\nabla u_\Tau)+\mu \BB w)v\right)\\
&= \int_{\omega_\Tau (S)} \AAA\nabla (u_\Tau-w)\cdot \nabla v +\sum_{i=1,2} \left(\int_{T_i} -Rv+\int_{T_i} \BB (\mu w-\lambda_\Tau u_\Tau)v\right)\\
&\lesssim \normNS{\nabla (u_\Tau-w)}\normNS{\nabla v}+ \normNS{R}\normNS{v}+\normNS{\mu w-\lambda_\Tau u_\Tau}\normNS{v}.
\end{align*}

For the second integral in~(\ref{E:split-jump}) we have
\begin{align*}
\int_S (\overline{J}-J) \overline{J} \varphi_S \leq \normS{\overline{J}\varphi_S}\normS{\overline{J}-J}
\leq \normS{\overline{J}}\normS{\overline{J}-J}.
\end{align*}

Hence, taking into account that $\normNS{\nabla v} \lesssim \frac{1}{h_T}\normNS{v}$,\quad $\normNS{v}\leq \normNS{\overline{J}}$ and $\normNS{\overline{J}}\lesssim h_T^{1/2}\normS{\overline{J}}$ we can write 
\begin{align*}
\normS{\overline{J}}^2 &\lesssim \normNS{\nabla (u_\Tau-w)}h_T^{-1/2}\normS{\overline{J}}+ \normNS{R}h_T^{1/2}\normS{\overline{J}} \\
&\quad+\normNS{\mu w-\lambda_\Tau u_\Tau}h_T^{1/2}\normS{\overline{J}}
+ \normS{\overline{J}}\normS{\overline{J}-J},
\end{align*}
and then
\begin{equation}\label{E:lowbound_aux4}
h_T^{1/2}\normS{\overline{J}} \lesssim \normNS{\nabla (u_\Tau-w)}+ h_T\normNS{R}+h_T\normNS{\mu w-\lambda_\Tau u_\Tau}
+ h_T^{1/2}\normS{\overline{J}-J}.
\end{equation}
Now, from~(\ref{E:lowbound_aux3}) and~(\ref{E:lowbound_aux4}) it follows
\begin{equation*}
h_T^{1/2}\normS{J}\lesssim \normNS{\nabla (u_\Tau-w)}+ h_T\normNS{R}+h_T\normNS{\mu w-\lambda_\Tau u_\Tau}
+ h_T^{1/2}\normS{J-\overline{J}}.
\end{equation*}
Adding the last equation over all $S\subset \partial T$, we obtain
\begin{equation*}
h_T^{1/2}\normbT{J} \lesssim \normNT{\nabla (u_\Tau-w)}+ h_T\normNT{R}+h_T\normNT{\mu w-\lambda_\Tau u_\Tau}
+ h_T^{1/2}\normbT{J-\overline{J}}.
\end{equation*}
The claim of this theorem follows by adding this last inequality and~(\ref{E:lowbound-int}).
\end{proof}

The next result is some kind of stability bound for the oscillation terms, which will be useful to obtain the bound of Corollary~\ref{C:lowerbound} which is what will be effectively	 used in our convergence proof.
\begin{lemma}\label{L:oscillation}
Under the assumptions of Theorem~\ref{T:general_lower_bound} there holds
$$h_T \normNT{R-\overline{R}}+h_T^{1/2} \normbT{J-\overline{J}}\lesssim h_T (2+\lambda_\Tau)\|u_\Tau\|_{H^1(\omega_\Tau(T))}.$$
\end{lemma}

\begin{proof}
\step{1} We first consider the term corresponding to the element residual.
\begin{align}
\normT{R-\overline{R}} &= \normT{-\nabla \cdot (\AAA\nabla u_\Tau)-\lambda_\Tau \BB u_\Tau +\overline{\nabla \cdot (\AAA\nabla u_\Tau)}+\lambda_\Tau \overline{\BB u_\Tau}} \notag\\ 
&\leq \normT{-\nabla \cdot (\AAA\nabla u_\Tau)+\overline{\nabla \cdot (\AAA\nabla u_\Tau)}}+\normT{\lambda_\Tau \big(\BB u_\Tau - \overline{\BB u_\Tau}\big)},\label{E:Proj_aux0}
\end{align}
where, as before, the bar denotes the $L^2(T)$-projection onto $\PP_{\ell-1}(T)$.

Let $A^T=(A^T_{ij})$ denote the mean value of $\AAA =(\AAA_{ij})$ over the element $T$, and note that $\nabla \cdot (A^T\nabla u_\Tau)=\overline{\nabla \cdot (A^T\nabla u_\Tau)}$. Thus, for the first term in the right hand side of~(\ref{E:Proj_aux0}) we have
\begin{align*}
\normT{-\nabla \cdot (\AAA\nabla u_\Tau)+\overline{\nabla \cdot (\AAA\nabla u_\Tau)}}&= \normT{\nabla \cdot \big((A^T-\AAA)\nabla u_\Tau\big)-\overline{\nabla \cdot \big((A^T-\AAA)\nabla u_\Tau\big)}}\\
&\leq \normT{\big(\nabla \cdot (A^T-\AAA)\big)\cdot\nabla u_\Tau-\overline{\big(\nabla\cdot(A^T-\AAA)\big)\cdot \nabla u_\Tau}}\\
& \quad+\normT{(A^T-\AAA):D^2 u_\Tau-\overline{(A^T-\AAA):D^2 u_\Tau}}\\
&\leq \normT{\big(\nabla \cdot(A^T-\AAA)\big)\cdot\nabla u_\Tau}+\normT{(A^T-\AAA):D^2 u_\Tau}\\
&\lesssim \|\AAA\|_{W^1_\infty(T)}\normT{\nabla u_\Tau}+\|A^T-\AAA\|_{L^\infty(T)}\normT{D^2 u_\Tau}.
\end{align*}
Since $
\|A^T-\AAA\|_{L^\infty(T)}\lesssim  h_T \|\AAA\|_{W^1_\infty(T)}$,
an inverse inequality leads to
\begin{equation*}
\normT{-\nabla \cdot (\AAA\nabla u_\Tau)+\overline{\nabla \cdot (\AAA\nabla u_\Tau)}} \leq  \|\AAA\|_{W^1_\infty(T)}\normT{\nabla u_\Tau} + h_T \|\AAA\|_{W^1_\infty(T)}\normT{D^2 u_\Tau}
\lesssim 
\normT{\nabla u_\Tau}.
\end{equation*}

For the second term in the right hand side of~(\ref{E:Proj_aux0}) we have
$$
\normT{\lambda_\Tau \big(\BB u_\Tau - \overline{\BB u_\Tau}\big)}\leq \normT{\lambda_\Tau \BB u_\Tau }\lesssim \lambda_\Tau\normT{u_\Tau},
$$
and therefore,
$$\normT{R-\overline{R}}\lesssim (1+\lambda_\Tau)\|u_\Tau\|_{H^1(T)}.$$
The same estimation holds for all elements in $\NT(T)$, and consequently, 
\begin{equation}\label{E:interior-osc}
h_T\normNT{R-\overline{R}}\lesssim h_T(1+\lambda_\Tau)\|u_\Tau\|_{H^1(\omega_\Tau(T))}.
\end{equation}

\step{2} Next, we analyze the jump residual. Let $S$ be a side of $T$ and let $T_1$ and $T_2$ denote the elements sharing $S$. Again, if the bar denotes the $L^2(S)$-projection onto $\PP_{\ell-1}(S)$, it follows that
\begin{align*}
\normS{J-\overline{J}} &=\normS{\sum_{i=1,2}(\AAA\nabla u_\Tau)|_{T_i}\cdot \overrightarrow{n_i}-\overline{\sum_{i=1,2}(\AAA\nabla u_\Tau)|_{T_i}\cdot \overrightarrow{n_i}}}.
\end{align*}

Using that $(A^{T_i}\nabla u_\Tau)|_{T_i}\cdot \overrightarrow{n_i}=\overline{(A^{T_i}\nabla u_\Tau)|_{T_i}\cdot \overrightarrow{n_i}}$ we have

\begin{align*}
\normS{J-\overline{J}} &=\normS{\sum_{i=1,2}\big((\AAA-A^{T_i})\nabla u_\Tau\big)|_{T_i}\cdot \overrightarrow{n_i}-\overline{\sum_{i=1,2}\big((\AAA-A^{T_i})\nabla u_\Tau\big)|_{T_i}\cdot \overrightarrow{n_i}}}\\
&\leq \normS{\sum_{i=1,2}\big((\AAA-A^{T_i})\nabla u_\Tau\big)|_{T_i}\cdot \overrightarrow{n_i}}\\
&\leq \sum_{i=1,2}\normS{\big((\AAA-A^{T_i})\nabla u_\Tau\big)|_{T_i}\cdot \overrightarrow{n_i}}\\
&\leq \sum_{i=1,2}\|\AAA|_{T_i}-A^{T_i}\|_{L^\infty(S)}\normS{\nabla u_\Tau|_{T_i}}\\
&\lesssim \sum_{i=1,2}h_T\|\AAA\|_{W^1_\infty(T_i)}h_T^{-1/2}\|u_\Tau\|_{H^1(T_i)}\\
&\lesssim h_T^{1/2}\|u_\Tau\|_{H^1(\omega_\Tau(S))}.
\end{align*}

Therefore, 
\begin{equation}\label{E:jump-osc}
h_T^{1/2}\normbT{J-\overline{J}}\lesssim h_T\|u_\Tau\|_{H^1(\omega_\Tau(T))}.
\end{equation}
Adding~(\ref{E:interior-osc}) and~(\ref{E:jump-osc}) we obtain the claim of this lemma.
\end{proof}


As an immediate consequence of Theorem~\ref{T:general_lower_bound} and Lemma~\ref{L:oscillation} the following result holds.

\begin{corollary}[Lower bound]\label{C:lowerbound}
Under the assumptions of Theorem~\ref{T:general_lower_bound} there holds
$$\est{\lambda_\Tau,u_\Tau}{T}\lesssim \normNT{\nabla (w-u_\Tau)} +h_T \normNT{\mu w}+h_T(1+\lambda_\Tau)\|u_\Tau\|_{H^1(\omega_\Tau(T))}.$$
\end{corollary}

\subsection{Adaptive loop}\label{S:adloop}

Our goal is to use an adaptive method to approximate the $j$-th eigenvalue and one of its eigenfunctions, for some fixed $j \in \NN$. From now on, we thus keep $j\in\NN$ fixed,  and let $\lambda$ denote the $j$-th eigenvalue of (\ref{E:cont-problem}) and $u$ an eigenfunction in $M(\lambda)$. 

The algorithm for approximating $\lambda$ and $M(\lambda)$ is an iteration of the following main steps:
\begin{enumerate}
 \item [(1)] $(\lambda_k,u_k)\definedas \textsf{SOLVE}(\VV_k)$.
 \item [(2)] $\{\eta_k(T)\}_{T\in\Tau_k}\definedas \textsf{ESTIMATE}(\lambda_k,u_k,\Tau_k)$.
 \item [(3)] $\MM_k\definedas \textsf{MARK}(\{\eta_k(T)\}_{T\in\Tau_k},\Tau_k)$.
 \item [(4)] $\Tau_{k+1}\definedas \textsf{REFINE}(\Tau_k,\MM_k)$, increment $k$.
\end{enumerate}

This is the same loop considered in~\cite{MSV_convergence}, the difference lies in the building blocks which we now describe in detail.

If $\Tau_k$ is a conforming triangulation of $\Omega$, the module \textsf{SOLVE} takes the space $\VV_k \definedas \VV_{\Tau_k}$ as input argument and outputs the $j$-th eigenvalue of the discrete problem~(\ref{E:disc-problem}) with $\Tau=\Tau_k$, i.e., $\lambda_k\definedas \lambda_{j,\Tau_k}$, and a corresponding eigenfunction $u_k\in\VV_k$. Therefore, $\lambda_k$ and $u_k$ satisfy
\begin{equation}\label{E:kdisc-problem}
\left\{
\begin{array}{l}
a(u_k,v_k)=\lambda_k \, b(u_k,v_k),\qquad \forall~v_k\in  \VV_k,\\
\normab{u_k}=1.
\end{array}
\right.
\end{equation}


Given $\Tau_k$ and the corresponding outputs $\lambda_k$ and $u_k$ of \textsf{SOLVE}, the module \textsf{ESTIMATE} computes and outputs the a posteriori error estimators $\{\eta_k(T)\}_{T\in\Tau_k}$, where
$$\eta_k(T)\definedas \eta_{\Tau_k}(\lambda_k,u_k;T).$$

Based upon the a posteriori error indicators $\{\eta_k(T)\}_{T\in\Tau_k}$, the module \textsf{MARK} collects elements of $\Tau_k$ in $\MM_k$. In order to simplify the presentation, the only requirement that we make on the module \textsf{MARK} is that the set of marked elements $\MM_k$ contains at least one element of $\Tau_k$ holding the largest value of estimator. That is, there exists one element $T_k^{\max} \in \MM_k$ such that
\[
 \eta_k(T_k^{\max}) = \max_{T \in \Tau_k} \eta_k(T).
\]
Whenever a marking strategy satisfies this assumption, we call it \emph{reasonable}, since this is what practitioners do in order to maximize the error reduction with a minimum effort. The most commonly used marking strategies, e.g., \textit{Maximum strategy} and \textit{Equidistribution strategy}, fulfill this condition, which is sufficient to guarantee that
\begin{equation}\label{E:marking}
 T \in \Tau_k \setminus \MM_k \qquad\Longrightarrow \qquad 
\eta_k(T) \lesssim \eta_k(\MM_k) \definedas \bigg( \sum_{T \in \MM_k} \eta_k(T)^2 \bigg)^{1/2}.
\end{equation}
This is slightly weaker, and is what we will use in our proof. The original \textit{D\"orfler's strategy} also guarantees~(\ref{E:marking}).

The refinement procedure \textsf{REFINE} takes the triangulation $\Tau_k$ and the subset $\MM_k\subset \Tau_k$ as input arguments. We require that all elements of $\MM_k$ are refined (at least once), and that a new conforming triangulation $\Tau_{k+1}$ of $\Omega$, which is a refinement of $\Tau_k$, is returned as output. 

In this way, starting with an initial conforming triangulation $\Tau_0$ of $\Omega$ and iterating the steps (1),(2),(3) and (4) of this algorithm, we obtain a sequence of successive conforming refinements of $\Tau_0$ called $\Tau_1,\Tau_2,\ldots$ and the corresponding outputs $(\lambda_k,u_k)$, $\{\eta_k(T)\}_{T\in\Tau_k}$, $\MM_k$ of the modules \textsf{SOLVE}, \textsf{ESTIMATE} and \textsf{MARK}, respectively.

For simplicity, we consider for the module \textsf{REFINE}, the concrete choice of the \textit{newest vertex} bisection procedure in two dimensions and the bisection procedure of Kossaczk\'y in three dimensions~\cite{Alberta}. Both these procedures refine the marked elements and some additional ones in order to keep conformity, and they also guarantee that
$$\kappa\definedas \sup_{k\in\NN_0} \kappa_{\Tau_k}<\infty,$$
i.e., $\{\Tau_k\}_{k\in\NN_0}$ is a sequence shape regular of triangulations of $\Omega$. 
It is worth mentioning that we do not assume \textsf{REFINE} to enforce the so-called \emph{interior node property}, and convergence is guaranteed nevertheless, this is an important difference with respect to~\cite{Giani}.

Regarding the module \textsf{MARK}, we stress that the marking is done only according to the error estimators; no marking due to oscillation is necessary, this is another important difference with respect to~\cite{Giani}, where the set of marked elements has to be enlarged so that D\"orfler's criterion is satisfied not only by the data oscillation terms, but also by the oscillation of the current solution $u_k$.


\section{Convergence to a limiting pair}\label{S:uinfty}

In this section we will prove that the sequence of discrete eigenpairs $\{(\lambda_k, u_k)\}_{k\in\NN_0}$ obtained by \textsf{SOLVE} throughout the adaptive loop of Section~\ref{S:adloop} has the following property: $\lambda_k$ converges to some $\lambda_\infty \in \RR$ and there exists a subsequence $\{u_{k_m}\}_{m\in\NN_0}$ of $\{u_k\}_{k\in\NN_0}$ converging in $H^1(\Omega)$ to a function $u_\infty$. 


Let us define the limiting space as $\VV_\infty\definedas\overline{\cup \VV_k}^{H^1_0(\Omega)}$, and note that $\VV_\infty$ is a closed subspace of $H^1_0(\Omega)$ and therefore, it is itself a Hilbert space with the inner product inherited from $H^1_0(\Omega)$.


Since $\Tau_{k+1}$ is always a refinement of $\Tau_k$, by the Minimum-Maximum principle $\{\lambda_k\}_{k\in\NN_0}$ is a decreasing sequence bounded below by $\lambda$. Therefore, there exists $\lambda_\infty>0$ such that
\begin{equation*}
\lambda_k \searrow \lambda_\infty.
\end{equation*}

From~(\ref{E:kdisc-problem}) it follows that 
\begin{equation}\label{E:normaa of u_k}
\normaa{u_k}^2=a(u_k,u_k)=\lambda_k b(u_k,u_k)=\lambda_k \normab{u_k}^2=\lambda_k \rightarrow \lambda_\infty,
\end{equation}
and therefore, that $\{u_k\}_{k\in\NN_0}$ is a bounded sequence in $\VV_\infty$. Then, there exists a subsequence $\{u_{k_m}\}_{m\in\NN_0}$ weakly convergent in $\VV_\infty$ to a function $u_\infty\in\VV_\infty$, so
\begin{equation}\label{E:u_k conv weakly in H1}
u_{k_m} \rightharpoonup u_\infty \quad \textrm{in}\quad H^1_0(\Omega).
\end{equation}
Using Rellich's theorem we can extract a subsequence of the last one, which we still denote $\{u_{k_m}\}_{m\in\NN_0}$, such that
\begin{equation}\label{E:u_k conv in L2}
u_{k_m} \longrightarrow u_\infty \quad \textrm{in}\quad L^2(\Omega).
\end{equation}

If $k_0\in\NN_0$ and $k_m\geq k_0$, for all $v_{k_0}\in\VV_{k_0}$ we have that
$
a(u_{k_m},v_{k_0})=\lambda_{k_m}b(u_{k_m},v_{k_0}),
$
, and when $m$ tends to infinity, we obtain that
$
a(u_\infty,v_{k_0})=\lambda_\infty b(u_\infty,v_{k_0})
$
. Since $k_0\in\NN_0$ and $v_{k_0}\in\VV_{k_0}$ are arbitrary we have that
\begin{equation}\label{E:infty equation}
a(u_\infty,v)=\lambda_\infty b(u_\infty,v), \qquad \forall~v\in \VV_\infty.
\end{equation}
On the other hand, since that $\normab{u_{k_m}}=1$, considering~(\ref{E:u_k conv in L2}) we conclude that
$\normab{u_\infty}=1.$
Now, taking into account~(\ref{E:infty equation}) we have that
$$\normaa{u_\infty}^2=\lambda_\infty \normab{u_\infty}^2=\lambda_\infty.$$ 
From (\ref{E:normaa of u_k}) it follows that 
$\normaa{u_{k_m}}^2=\lambda_{k_m} \longrightarrow \lambda_\infty,$
and therefore,
$\normaa{u_{k_m}} \rightarrow \normaa{u_\infty}.$ This, together with~(\ref{E:u_k conv weakly in H1}) yields
\begin{equation*}
u_{k_m} \longrightarrow u_\infty \quad \textrm{in} \quad H^1_0(\Omega).
\end{equation*}

Summarizing, we have proved the following
\begin{theorem}\label{T:limiting function}
There exist $\lambda_\infty\in\RR$ and $u_\infty\in\VV_\infty$ such that 
\begin{equation*}
\left\{
\begin{array}{l}
a(u_\infty,v)=\lambda_\infty \, b(u_\infty,v),\qquad \forall~v\in  \VV_\infty,\\
\normab{u_\infty}=1.
\end{array}
\right.
\end{equation*}
Moreover, $\D\lambda_\infty = \lim_{k\to\infty} \lambda_k$ and there exists a subsequence $\{u_{k_m}\}_{m\in\NN_0}$ of $\{u_k\}_{k\in\NN_0}$ such that
$$u_{k_m} \longrightarrow u_\infty \quad \textrm{in} \quad H^1_0(\Omega).$$
\end{theorem}

\begin{remark}\label{R:subsubsequence}
 It is important to notice that from any subsequence $\{(\lambda_{k_m},u_{k_m})\}_{m\in\NN_0}$ of $\{(\lambda_k,u_k)\}_{k\in\NN_0}$, we can extract another subsequence $\{(\lambda_{k_{m_n}},u_{k_{m_n}})\}_{n\in\NN_0}$, such that $u_{k_{m_n}}$ converges in $H^1(\Omega)$ to some function $\tilde u_\infty\in\VV_\infty$ that satisfies
\begin{equation*}
\left\{
\begin{array}{l}
a(\tilde u_\infty,v)=\lambda_\infty \, b(\tilde u_\infty,v),\qquad \forall~v\in  \VV_\infty,\\
\normab{\tilde u_\infty}=1.
\end{array}
\right.
\end{equation*}
\end{remark}

\section{Convergence of estimators}\label{S:convest}

In this section we will prove that the global a posteriori estimator defined in Section~\ref{S:apost} tends to zero. We will follow the same steps as in~\cite{MSV_convergence} providing the proofs of the results that are problem dependent. Those geometrical results that are consequences of the fact that  we are only refining will be stated without proof, but with a precise reference to the result from~\cite{MSV_convergence} being used.

In order not to clutter the notation, we will still call  $\{u_k\}_{k\in\NN_0}$ to the subsequence $\{u_{k_m}\}_{m\in\NN_0}$, and $\{\Tau_k\}_{k\in\NN_0}$ to the sequence $\{\Tau_{k_m}\}_{m\in\NN_0}$. Also, we will replace the subscript $\Tau_k$ by $k$ (e.g.\ $\Nk(T)\definedas \NTk(T)$ and $\omega_k(T)\definedas \omega_{\Tau_k}(T)$), and 
 whenever $\Xi$ is a subset of $\Tau_k$, $\eta_k(\Xi)^2$ will denote the sum $\sum_{T\in \Xi}\eta_k(T)^2$.%

The main result of this section is the following
\begin{theorem}[Estimator's convergence]\label{T:est_conv}
If $\{\Tau_k\}_{k\in\NN_0}$ denote the triangulations corresponding to the convergent subsequence of discrete eigenpairs from Theorem~\ref{T:limiting function}, then
$$\lim_{k\rightarrow\infty} \eta_k(\Tau_k)=0.$$
\end{theorem}

In order to prove this theorem we consider the following decomposition of $\Tau_k$, which was first established in~\cite{MSV_convergence}.

\begin{definition}\label{D:splitting}
Given the sequence $\{\Tau_k\}_{k\in\NN_0}$ of triangulations, for each $k \in \NN_0$ we define the following (disjoint) subsets of $\Tau_k$.
\begin{itemize}
 \item $\Tau_k^0\definedas \{T\in\Tau_k : T'~ \textrm{is refined at least} ~n_d~ \textrm{times, for all}~ T'\in \Nk(T)\}$;
 \item $\Tau_k^+\definedas \{T\in\Tau_k : T'~ \textrm{is never refined, for all}~T'\in \Nk(T)\}$;
 \item $\Tau_k^*\definedas \Tau_k \setminus (\Tau_k^0\cup \Tau_k^+)$.
\end{itemize}
We also define the three (overlapping) regions in $\Omega$:
\begin{itemize}
 \item $\Omega_k^0\definedas \bigcup\limits_{T\in\Tau_k^0} \omega_k(T)$;
 \item $\Omega_k^+\definedas \bigcup\limits_{T\in\Tau_k^+} \omega_k(T)$;
 \item $\Omega_k^*\definedas \bigcup\limits_{T\in\Tau_k^*} \omega_k(T)$.
\end{itemize}
\end{definition}

We will prove that $\eta_k(\Tau_k^0)$, $\eta_k(\Tau_k^*)$ and $\eta_k(\Tau_k^+)$ tend to zero as $k$ tends to infinity in Theorems~\ref{T:est_conv_1}, \ref{T:est_conv_2} and \ref{T:est_conv_3}. Since $\eta_k(\Tau_k)^2 = \eta_k(\Tau_k^0)^2 + \eta_k(\Tau_k^+)^2 + \eta_k(\Tau_k^*)^2 $, Theorem~\ref{T:est_conv} will follow from these results.

\begin{definition}[Meshsize function] 
We define $h_k\in L^\infty(\Omega)$ as the piecewise constant function
$$h_k|_T\definedas|T|^{1/d},\qquad \forall~T\in\Tau_k.$$
\end{definition}

For almost every $x\in\Omega$ there holds that $h_k(x)$ is monotonically decreasing and bounded from below by $0$. Therefore,
$$h_\infty(x)\definedas\lim\limits_{k\rightarrow\infty} h_k(x)$$
is well-defined for almost every $x\in\Omega$ and defines a function in $L^\infty(\Omega)$. Moreover, the following result holds~\cite[Lemma 4.3 and Corollary 4.1]{MSV_convergence}.
\begin{lemma}\label{L:hktendstozero}
The sequence $\{h_k\}_{k\in\NN_0}$ converges to $h_\infty$ uniformly, i.e.,
$$\lim_{k\rightarrow\infty} \|h_k-h_\infty\|_{L^\infty(\Omega)}=0,$$
and if $\chi_{\Omega_k^0}$ denotes the characteristic function of $\Omega_k^0$ then
$$\lim_{k\rightarrow\infty} \|h_k \chi_{\Omega_k^0}\|_{L^\infty(\Omega)}=0.$$
\end{lemma}

This lemma is a consequence of the fact that the sequence of triangulations is obtained by refinement only, and that every time an element $T \in \Tau_k$ is refined into $\Tau_{k+1}$, $h_{k+1}(x)\leq \left(\frac{1}{2}\right)^{1/d} h_k(x)$ for almost every $x\in T$. But it is otherwise independent of the marking strategy. The next result is also independent of the marking strategy, it is just a consequence of the fact that $u_k\rightarrow u_\infty$, the lower bound and the convergence of $\|h_k \chi_{\Omega_k^0}\|_{L^\infty(\Omega)}$ to zero.

\begin{theorem}[Estimator's convergence: First part]\label{T:est_conv_1}
If $\{\Tau_k\}_{k\in\NN_0}$ denote the triangulations corresponding to the convergent subsequence of discrete eigenpairs from Theorem~\ref{T:limiting function}, then the contribution of $\Tau_k^0$ to the estimator vanishes in the limit, i.e.,
$$\lim_{k\rightarrow\infty} \eta_k(\Tau_k^0)= 0.$$
\end{theorem}

\begin{proof}
Using Corollary~\ref{C:lowerbound} with $\WW=\VV_\infty$, $w=u_\infty$, and $\mu=\lambda_\infty$ we have that
\begin{align*}
\eta_k(\Tau_k^0)^2 &= \sum_{T\in \Tau_k^0} \eta_k(T)^2\\
&\lesssim\sum_{T\in \Tau_k^0} \normkNT{\nabla(u_k-u_\infty)}^2+h_T^2 \normkNT{\lambda_\infty u_\infty}^2+h_T^2(1+\lambda_k)^2\|u_k\|_{H^1(\omega_k(T))}^2\\
&\lesssim\norm{\nabla(u_k-u_\infty)}^2+\|h_k \chi_{\Omega_k^0}\|_{L^\infty(\Omega)}^2\big(\norm{\lambda_\infty u_\infty}^2+(1+\lambda_k)^2\|u_k\|_{H^1(\Omega)}^2\big)
\end{align*}
Since $\lambda_k \rightarrow \lambda_\infty$ in $\RR$ and $u_k \rightarrow u_\infty$ in $H^1(\Omega)$, Lemma~\ref{L:hktendstozero} implies the claim.
\end{proof}

The following lemma was proved as the first step of the proof of~\cite[Proposition 4.2]{MSV_convergence}, it is also a consequence of the fact that the sequence of triangulations is obtained by refinement, without coarsening, and it is independent of the specific problem being considered.
\begin{lemma}\label{L:interfasetendstozero}
If $\Omega_k^*$ is as in Definition~\ref{D:splitting}, then
$$\lim_{k\rightarrow\infty}|\Omega_k^*|= 0.$$
\end{lemma}

From Corollary~\ref{C:lowerbound}, Lemma~\ref{L:interfasetendstozero} and the fact that $u_k\rightarrow u_\infty$ en $H^1$ we obtain
\begin{theorem}[Estimator's convergence: Second part]\label{T:est_conv_2}
If $\{\Tau_k\}_{k\in\NN_0}$ denote the triangulations corresponding to the convergent subsequence of discrete eigenpairs from Theorem~\ref{T:limiting function}, then the contribution of $\Tau_k^*$ to the estimator vanishes in the limit, i.e.,
$$\lim_{k\rightarrow\infty}\eta_k(\Tau_k^*)= 0.$$
\end{theorem}

\begin{proof}
Let $(\lambda, u)$ be any eigenpair of~(\ref{E:cont-problem}). Then Corollary~\ref{C:lowerbound} with $\WW=H^1_0(\Omega)$, $w=u$, and $\mu=\lambda$ implies that
\begin{align*}
\eta_k(\Tau_k^*)^2 &= \sum_{T\in \Tau_k^*} \eta_k(T)^2\\
&\lesssim\sum_{T\in \Tau_k^*} \normkNT{\nabla(u_k-u)}^2+h_T^2 \normkNT{\lambda u}^2+h_T^2(1+\lambda_k)^2\| u_k\|_{H^1(\omega_k(T))}^2\\
&\lesssim \|\nabla(u_k-u)\|_{\Omega_k^*}^2+\lambda^2\|u\|_{\Omega_k^*}^2+ (1+\lambda_k)^2\|u_k\|_{H^1(\Omega_k^*)}^2\\
&\lesssim \|\nabla(u_k-u_\infty)\|_{\Omega}^2+\|\nabla(u_\infty-u)\|_{\Omega_k^*}^2+\lambda^2\|u\|_{\Omega_k^*}^2\\
&\quad+(1+\lambda_k)^2\|u_k-u_\infty\|_{H^1(\Omega)}^2+(1+\lambda_k)^2\|u_\infty\|_{H^1(\Omega_k^*)}^2.
\end{align*}
Taking into account that $\lambda_k \to \lambda_\infty$ in $\RR$, $u_k \rightarrow u_\infty$ in $H^1(\Omega)$ and Lemma~\ref{L:interfasetendstozero}, the claim follows.
\end{proof}

In order to prove that the estimator contribution from $\Tau_k^+$ vanishes in the limit, we make the following 

\begin{definition}\label{D:Tau-plus}
Let $\Tau^+$ be the set of elements that are never refined, i.e.,
$$\Tau^+\definedas \bigcup_{k\geq 0}\bigcap_{m \geq k} \Tau_m,$$
and let the set $\Omega^+$ be defined as
$$\Omega^+\definedas \bigcup_{T\in\Tau^+} T.$$
\end{definition}

It is interesting to observe at this point that
\begin{lemma}\label{L:h0<=>omega+}
The set $\Omega^+$ is empty if and only if $\D \lim_{k\rightarrow\infty} \|h_k\|_{L^\infty(\Omega)}=0$.
\end{lemma}

\begin{proof}
If $\Omega^+$ is empty, then $\Omega^+_k$ and $\Omega^*_k$ are empty for all $k\in\NN_0$, and $\|h_k\|_{L^\infty(\Omega)}=\|h_k\|_{L^\infty(\Omega^0_k)}$ which tends to zero by Lemma~\ref{L:hktendstozero}.
Conversely, if $\lim_{k\rightarrow\infty} \|h_k\|_{L^\infty(\Omega)}=0$, then $\Omega^+$ must be empty, otherwise there would exist $T\in\Tau^+$ and for all $k$ we would have $\|h_k\|_{L^\infty(\Omega)}\geq |T|^{1/d}$.
\end{proof}

This lemma, as Lemma~\ref{L:hktendstozero}, is just a geometric observation, and a consequence of the fact that the sequence of triangulations is shape regular and obtained by refinement, but it is independent of the particular problem being considered.

As an immediate consequence of Definition~\ref{D:Tau-plus} and Lemma 4.1 in \cite{MSV_convergence} we have that
$$\Tau^+= \bigcup_{k\geq 0} \Tau_k^+.$$

\begin{remark}
Theorems~\ref{T:est_conv_1} and~\ref{T:est_conv_2} hold independently of the marking strategy. In the next theorem, we will make use for the first time of the assumption~\eqref{E:marking} done on the module \textsf{MARK}.

\end{remark}

\begin{theorem}[Estimator's convergence: Third part]\label{T:est_conv_3}
If $\{\Tau_k\}_{k\in\NN_0}$ denote the triangulations corresponding to the convergent subsequence of discrete eigenpairs from Theorem~\ref{T:limiting function}, then the contribution of $\Tau_k^+$ to the estimator vanishes in the limit, i.e.,
$$\lim_{k\rightarrow\infty} \eta_k(\Tau_k^+)=0.$$
\end{theorem}

\begin{proof}
Let $T\in\Tau^+$, then there exists $k_0$ such that $T\in \Tau_k$, for all $k\geq k_0$. Taking into account that all marked elements are at least refined once, we have that $T \notin \MM_k$. From assumption~\eqref{E:marking}, $\eta_k(T) \lesssim \eta_k(\MM_k)$. Since $\MM_k \subset \Tau_k^* \cup \Tau_k^0$, Theorems~\ref{T:est_conv_1} and~\ref{T:est_conv_2} imply that
\[
 \eta_k(T)^2 \lesssim \eta_k(\MM_k)^2 \le \eta_k(\Tau_k^*)^2 + \eta_k(\Tau_k^0)^2 \longrightarrow 0.
\]
We have thus proved that 
\begin{equation*}
\eta_k(T)\longrightarrow 0, \qquad\text{for all $T\in\Tau^+$.}
\end{equation*}

Now, we will prove that, moreover, 
$$\sum_{T\in \Tau_k^+} \eta_k(T)^2\longrightarrow 0.$$
To prove this, we resort to a generalized majorized convergence theorem. We first define 
\begin{equation*}
\epsilon_k|_T\definedas\frac{1}{|T|}\eta_k(T)^2, \quad\text{for all $T\in\Tau_k^+$,}\qquad\text{and}\quad \epsilon_k\definedas 0,\quad\text{otherwise}.
\end{equation*}
 Then $\sum_{T\in \Tau_k^+} \eta_k(T)^2 = \int_{\Omega} \epsilon_k(x) \, dx$, and $\epsilon_k(x) \to 0$ as $k\to\infty$ for almost every $x\in\Omega$.
It remains to prove that  $\int_{\Omega} \epsilon_k(x) \, dx \to 0$ as $k\to\infty$.

Let $k$ be fixed. Due to the definition of $\Tau_k^+$, for $T\in\Tau_k^+$ we have that $\omega_k(T)=\omega_j(T)$ for all $j\geq k$, and we can drop the subscript and call this set $\omega(T)$. Using Corollary~\ref{C:lowerbound} we have that if $(\lambda,u)$ is any fixed eigenpair of~(\ref{E:cont-problem}),
\begin{align*}
\eta_k(T)^2&\lesssim\|\nabla(u_k-u)\|_{\omega(T)}^2+\|\lambda u\|_{\omega(T)}^2 + (1+\lambda_k)^2\|u_k\|_{H^1(\omega(T))}^2\\
&\lesssim\|\nabla(u_k-u_\infty)\|_{\omega(T)}^2+\|\nabla u_\infty\|_{\omega(T)}^2+\|\nabla u\|_{\omega(T)}^2+\|\lambda u\|_{\omega(T)}^2 
\\
&\quad + (1+\lambda_0)^2\|u_k-u_\infty\|_{H^1(\omega(T))}^2+ (1+\lambda_0)^2\|u_\infty\|_{H^1(\omega(T))}^2\\
&\lesssim (1+\lambda_0)^2\left(\|u_k-u_\infty\|_{H^1(\omega(T))}^2 +c_T^2\right),
\end{align*}
where
$$c_T^2\definedas \|u_\infty\|_{H^1(\omega(T))}^2+\|\lambda u\|_{\omega(T)}^2+\|\nabla u\|_{\omega(T)}^2,$$
and fulfills
\begin{equation}\label{E:c_T}
\sum_{T\in\Tau_k^+} c_T^2\lesssim\|u_\infty\|_{H^1(\Omega)}^2+\norm{\lambda u}^2+\norm{\nabla u}^2<\infty.
\end{equation}

Let now $M_k$ be defined by
\begin{equation*}
M_k|_T\definedas\frac{C}{|T|}\big(\|u_k-u_\infty\|_{H^1(\omega(T))}^2 +c_T^2\big),
\quad\text{for all $T\in\Tau_k^+$,}\qquad\text{and}\quad M_k\definedas 0 ,\quad\text{otherwise},
\end{equation*}
where $C$ is chosen so that $0\leq\epsilon_k(x)\leq M_k(x)$, for all $x\in\Omega$. If we define
\begin{equation*}
M|_T\definedas C\frac{c_T^2}{|T|},
\quad\text{for all $T\in\Tau^+$,} \qquad \text{and} \quad M\definedas 0 , \quad\text{otherwise},
\end{equation*}
then
\begin{align*}
\int_{\Omega^+} |M_k(x)-M(x)|~dx &=\sum_{T\in\Tau^+\setminus \Tau^+_k} \int_T |M_k(x)-M(x)|~dx+\sum_{T\in\Tau^+_k} \int_T |M_k(x)-M(x)|~dx\\
&=\sum_{T\in\Tau^+\setminus \Tau^+_k} \int_T |M(x)|~dx+C\sum_{T\in\Tau^+_k} \|u_k-u_\infty\|_{H^1(\omega(T))}^2\\
&\lesssim C\sum_{T\in\Tau^+\setminus \Tau^+_k} c_T^2+C\|u_k-u_\infty\|_{H^1(\Omega)}^2.
\end{align*}
The terms in the right hand side tend to zero when $k$ tends to infinity, due to~(\ref{E:c_T}) and the fact that $u_k$ converges to $u_\infty$ in $H^1_0(\Omega)$. Therefore, 
$$M_k\longrightarrow M,\qquad \textrm{in}\quad L^1(\Omega^+).$$
Hence, using that $\epsilon_k(x)\rightarrow 0$, for almost every $x\in\Omega$, we can apply a generalized majorized convergence theorem \cite[p.1015]{Zeidler} to conclude that
$$
\eta_k(\Tau_k^+)^2 = \sum_{T\in \Tau_k^+} \eta_k(T)^2=\int_{\Omega^+} \epsilon_k(x)~dx \longrightarrow 0,
$$
as $k\rightarrow \infty$.
\end{proof}

We have proved in this section that $\eta_k(\Tau_k)\rightarrow 0$ as $k\rightarrow \infty$. In the next section we will use this result to conclude that $(\lambda_\infty, u_\infty)$ is an eigenpair of the continuous problem~(\ref{E:cont-problem}).

%
%
%
%

\section{The limiting pair is an eigenpair}\label{S:eigenfunction}

%
%
%
In this section we will prove that $(\lambda_\infty,u_\infty)$ is an eigenpair of the continuous problem~(\ref{E:cont-problem}). 
The idea in~\cite{MSV_convergence} to prove that $u_\infty$ is the exact solution to the continuous problem, consisted in using the \emph{reliability} of the a posteriori error estimators, that is, the fact that the error in energy norm is bounded (up to a constant) by the global error estimator. Such a bound does not hold in this case unless the underlying triangulation is sufficiently fine (see Theorem~\ref{T:upperbound}). We do not enforce such a condition on the initial triangulation $\Tau_0$, since the term \emph{sufficiently fine} is not easily quantifiable. Instead we resort to another idea, we will bound $a(u_\infty,v)-\lambda_\infty b(u_\infty,v)$ by the residuals of the discrete problems, which are in turn bounded by the estimators, and were proved to converge to zero in the previous section.

\begin{theorem}\label{T:eigenfunction}
The limiting pair $(\lambda_\infty,u_\infty)$ of Theorem~\ref{T:limiting function} is an eigenpair of the continuous problem~\eqref{E:cont-problem}. That is,
\begin{equation*}
\left\{
\begin{array}{l}
a(u_\infty,v)=\lambda_\infty \, b(u_\infty,v),\qquad \forall~v\in  H^1_0(\Omega),\\
\normab{u_\infty}=1.
\end{array}
\right.
\end{equation*}
\end{theorem}

\begin{proof}
We know that $\normab{u_\infty}=1$ due to Theorem~\ref{T:limiting function}. It remains to prove that
$$a(u_\infty,v)=\lambda_\infty b(u_\infty,v),\qquad \forall~v\in  H^1_0(\Omega).$$
Let $v\in  H^1_0(\Omega)$, and let $v_k\in \VV_k$ be the Scott-Zhang interpolant~\cite{Scott-Zhang},\cite{Scott-Zhang92} of $v$, which satisfies
$$\normT{v-v_k}\lesssim h_T \|\nabla v\|_{\omega_k(T)}\qquad \textrm{and}\qquad \normbT{v-v_k}\lesssim h_T^{1/2} \|\nabla v\|_{\omega_k(T)}.$$
From (\ref{E:kdisc-problem}) we have
$$a(u_k,v_k)=\lambda_k b(u_k,v_k),$$
for all $k$, and then
\begin{align}
\notag
  |a(u_\infty,v)-&\lambda_\infty b(u_\infty,v)|= |a(u_\infty,v)-\lambda_\infty b(u_\infty,v)-a(u_k,v_k)+\lambda_k b(u_k,v_k)| 
\\
\notag
  &= |a(u_k,v-v_k)-\lambda_k b(u_k,v-v_k)+b(\lambda_k u_k-\lambda_\infty u_\infty,v)+a(u_\infty-u_k,v)| \\
\label{E:to-be-bounded}
  &\le |a(u_k,v-v_k)-\lambda_k b(u_k,v-v_k)|+|b(\lambda_k u_k-\lambda_\infty u_\infty,v)|+|a(u_\infty-u_k,v)|.
\end{align}

The second term in~\eqref{E:to-be-bounded} can be bounded as 
\begin{align*}
|b(\lambda_k u_k-\lambda_\infty u_\infty,v)|&=|\lambda_k b(u_k- u_\infty,v)+(\lambda_k-\lambda_\infty)b(u_\infty,v)|\\
&\leq |\lambda_k| |b(u_k- u_\infty,v)|+|\lambda_k-\lambda_\infty||b(u_\infty,v)|\\
&\lesssim \lambda_0 \norm{u_k- u_\infty}\norm{v}+|\lambda_k-\lambda_\infty|\norm{u_\infty}\norm{v}\\
&\lesssim\left(\lambda_0 \norm{u_k- u_\infty}+|\lambda_k-\lambda_\infty| \norm{u_\infty}\right)\norm{v}.
\end{align*}

And the third term in~\eqref{E:to-be-bounded} is bounded by 
$$|a(u_\infty-u_k,v)| \lesssim \norm{\nabla (u_\infty-u_k)}\norm{\nabla v}.$$

Finally, the first term in~\eqref{E:to-be-bounded} can be bounded following the steps of the proof of the a posteriori upper bound, as follows:
\begin{align*}
|a(u_k,v-v_k)-\lambda_k b(u_k,v-v_k)|&=\left|\sum_{T\in\Tau_k} \int_T \AAA \nabla u_k\cdot \nabla (v-v_k) -\lambda_k \int_T \BB u_k(v-v_k)\right|\\
&=\left|\sum_{T\in\Tau_k} \int_T \big(-\nabla\cdot (\AAA \nabla u_k)-\lambda_k\BB u_k\big)(v-v_k) +\int_{\partial T} (v-v_k)\AAA\nabla u_k\cdot \overrightarrow{n}\right|\\
&=\left|\sum_{T\in\Tau_k} \int_T R(\lambda_k,u_k)(v-v_k) +\frac{1}{2}\int_{\partial T} (v-v_k)J(u_k)\right|,
\end{align*}
with $R(\lambda_k,u_k)$ and $J(u_k)$ as defined in~(\ref{E:element-residual}) and~(\ref{E:jump-residual}). Now, by H\"older and Cauchy-Schwarz inequalities we obtain
\begin{align*}
|a(u_k,v-v_k)-\lambda_k b(u_k,v-v_k)|&\leq\sum_{T\in\Tau_k}\normT{R(\lambda_k,u_k)}\normT{v-v_k}+\normbT{J(u_k)}\normbT{v-v_k}\\
&\lesssim\sum_{T\in\Tau_k}\normT{R(\lambda_k,u_k)}h_T\|\nabla v\|_{\omega_k(T)}+\normbT{J(u_k)}h_T^{1/2}\|\nabla v\|_{\omega_k(T)}\\
&\lesssim \left(\sum_{T\in\Tau_k}h_T^2\normT{R(\lambda_k,u_k)}^2+h_T\normbT{J(u_k)}^2\right)^{1/2}\norm{\nabla v}\\
&=\eta_k(\Tau_k)\norm{\nabla v}.
\end{align*}

Summarizing, we have that
\[
 |a(u_\infty,v)-\lambda_\infty b(u_\infty,v)|
\lesssim \left((1+\lambda_0)\left\|u_k- u_\infty\right\|_{H^1(\Omega)}+|\lambda_k-\lambda_\infty|\norm{u_\infty}+ 
\eta_k(\Tau_k)\right) \left\|v\right\|_{H^1(\Omega)}.
\]
Using the convergence of $u_k$ to $u_\infty$ in $H^1(\Omega)$ and $\lambda_k$ to $\lambda_\infty$ in $\RR$ from Theorem~\ref{T:limiting function}, and the convergence of the global estimator to zero from Theorem~\ref{T:est_conv}, we conclude that
$$|a(u_\infty,v)-\lambda_\infty b(u_\infty,v)|=0,$$
and the proof is completed.
\end{proof}

\section{Main result and concluding remarks}\label{S:main-result}

We conclude this article by stating and proving our main result, which is a consequence of the results in the previous sections, and discussing its strengths and weaknesses.

\begin{theorem}\label{T:cuasi main result}
Let $\{(\lambda_k,u_k)\}_{k\in\NN_0}$ denote the whole sequence of discrete eigenpairs obtained through the adaptive loop stated in Section~\ref{S:adloop}. Then, there exists an eigenvalue $\lambda$ of the continuous problem~(\ref{E:cont-problem}) such that
\begin{equation*}
\lim_{k\rightarrow\infty} \lambda_k=\lambda \qquad \textrm{and}\qquad \lim_{k\rightarrow\infty} \dist_{H^1_0(\Omega)}(u_k,M(\lambda)) = 0.
\end{equation*}
\end{theorem}

\begin{proof}
By Theorem~\ref{T:limiting function}, taking $\lambda\definedas \lambda_\infty$, we have that $\lim_{k\rightarrow\infty} \lambda_k=\lambda,$
and by Theorem~\ref{T:eigenfunction}, $\lambda$ is an eigenvalue of the continuous problem~\eqref{E:cont-problem}. In order to prove that $\D \lim_{k\rightarrow\infty} \dist_{H^1_0(\Omega)}(u_k,M(\lambda)) = 0$ we argue by contradiction. If the result were not true, then there would exist a number $\epsilon>0$ and a subsequence $\{u_{k_m}\}_{m\in\NN_0}$ of $\{u_k\}_{k\in\NN_0}$ such that
\begin{equation}\label{E:Absurd}
\dist_{H^1_0(\Omega)}(u_{k_m},M(\lambda))>\epsilon,\qquad \forall~m\in\NN_0.
\end{equation}
By Remark~\ref{R:subsubsequence} it is possible to extract a subsequence of $\{u_{k_m}\}_{m\in\NN_0}$ which still converges to some function $\tilde u_\infty \in \VV_\infty$. By the arguments of Sections~\ref{S:convest} and~\ref{S:eigenfunction}, $\tilde u_\infty$ is an eigenfunction of the continuous problem~\eqref{E:cont-problem} corresponding to the same eigenvalue $\lambda$. That is, a subsequence of $\{u_{k_m}\}_{m\in\NN_0}$ converges to an eigenfunction in $M(\lambda)$, this contradicts~(\ref{E:Absurd}) and completes the proof.
\end{proof}

\begin{remark}
We have proved that the discrete eigenvalues converge to an eigenvalue of the continuous problem, and the discrete eigenfunctions converge to the set of the corresponding continuous eigenfunctions, and this is the main result of this article. But there is still an open question: If $\lambda_k$ was chosen as the $j$-th eigenvalue of the discrete problem over $\Tau_k$, is it true that $\{\lambda_k\}_{k\in\NN_0}$ converges to the $j$-th eigenvalue of the continuous problem? The answer is affirmative for a large number of problems, but not necessarily for all. There could be some pathological cases in which looking for the $j$-th eigenvalue we converge to one that is larger.
\end{remark}

 We now state an assumption on problem~(\ref{E:cont-problem}) that we will prove to be sufficient to guarantee that the convergence holds to the desired eigenvalue/eigenfunction. More precise sufficient conditions on problem data $\AAA$ and $\BB$ to guarantee that this assumption holds will be stated below.

\begin{assumptions}[Non-Degeneracy Assumption]\label{A:non-deg}
We will say that problem~(\ref{E:cont-problem}) satisfies the \emph{Non-Degeneracy Assumption} if whenever $u$ is an eigenfunction of~(\ref{E:cont-problem}), there is no nonempty open subset $\OO$ of $\Omega$ such that $u|_\OO\in\PP_\ell(\OO)$.
\end{assumptions}

\begin{theorem}\label{T:main-result}
Let us suppose that the continuous problem~(\ref{E:cont-problem}) satisfies the Non-Degeneracy Assumption~\ref{A:non-deg}, and let $\{(\lambda_k,u_k)\}_{k\in\NN_0}$ denote the whole sequence of discrete eigenpairs obtained through the adaptive loop stated in Section~\ref{S:adloop} and $\lambda$ denote the $j$-th eigenvalue of the continuous problem~(\ref{E:cont-problem}). Then,
\begin{equation*}
\lim_{k\rightarrow\infty} \lambda_k=\lambda \qquad \textrm{and}\qquad \lim_{k\rightarrow\infty} \dist_{H^1_0(\Omega)}(u_k,M(\lambda)) = 0.
\end{equation*}
\end{theorem}

Before embarking into the proof of this theorem, it is worth mentioning that the model case of $\AAA \equiv I$ and $\BB \equiv 1$ satisfies Assumption~\ref{A:non-deg}, due to the fact that the eigenfunctions of the laplacian are analytic. A weaker assumption on the coefficients $\AAA$ and $\BB$ that guarantee non-degeneracy of the problem are given in the following

\begin{lemma}
If $\AAA$ is continuous, and piecewise $\PP_1$, and $\BB$ is piecewise constant, then problem~(\ref{E:cont-problem}) satisfies the Non-Degeneracy Assumption~\ref{A:non-deg}.
\end{lemma}

\begin{proof}
We will argue by contradiction. Let us suppose that there exists an eigenfunction $u$ of~(\ref{E:cont-problem}) with corresponding eigenvalue $\lambda$, and a nonempty open subset $\OO$ of $\Omega$ such that $u|_\OO \in \PP_\ell(\OO)$. Without loss of generality, we may assume that $\AAA|_\OO \in \PP_1(\OO)$ and $\BB$ is constant over $\OO$. Then
$$-\nabla \cdot (\AAA\nabla u)=\lambda \BB u, \qquad \textrm{in}~\OO.$$
Since $u|_\OO\in\PP_\ell(\OO)$, we have that $-\nabla \cdot (\AAA\nabla u)\in \PP_{\ell-1}(\OO)$, and the last equation implies that $u|_\OO\in \PP_{\ell-1}(\OO)$. Repeating this argument we finally obtain that 
$$u|_\OO\equiv 0,$$
which cannot be true. In fact, $u$ is a solution of a linear elliptic equation of second order with uniformly elliptic and Lipschitz leading coefficients and therefore, it cannot vanish in an open subset of $\Omega$ unless it vanishes over $\Omega$~\cite{Han}.
\end{proof}

\begin{remark}
Searching for other sufficient conditions on the coefficients to guarantee Assumption~\ref{A:non-deg} is out of the scope of this article. We believe that in the assumptions of the previous lemma, $\AAA$ can be allowed to be piecewise continuous with discontinuities along Lipschitz interfaces. The only thing needed is a proof of the fact that solutions to elliptic problems with coefficients like these cannot vanish in an open subset of $\Omega$ unless they vanish over all $\Omega$. We conjecture that this could be proved using Han's result~\cite{Han} in combination with Hopf's lemma~\cite{Gilbarg-Trudinger}, but it will be subject of future work.
\end{remark}

We now proceed to prove Theorem~\ref{T:main-result}, which will be a consequence of the following lemma.

\begin{lemma}\label{L:nondeg->h0}
Let $\{ h_k \}_{k\in\NN_0}$ denote the sequence of meshsize functions obtained through the adaptive loop stated in Section~\ref{S:adloop}. If the continuous problem~(\ref{E:cont-problem}) satisfies the \emph{Non-Degeneracy Assumption}~\ref{A:non-deg}, then $\| h_k \|_{L^\infty(\Omega)} \to 0$ as $k \to \infty$.
\end{lemma}

\begin{proof}
We argue by contradiction. By Lemma~\ref{L:h0<=>omega+}, if $\| h_k \|_{L^\infty(\Omega)}$ does not tend to zero, then $\Omega^+$ is not empty, and then there exists $T\in\Tau^+$, and thus $k_0\in\NN_0$ such that $T\in\Tau_k$, for all $k\geq k_0$. Since $\|u_{k_m}-u_\infty\|_{L^2(T)}\rightarrow 0$ as $m\rightarrow \infty$, and $u_k|_T \in\PP_\ell(T)$, for all $k\geq 0$, using that $\PP_\ell(T)$ is a finite dimensional space we conclude that
\begin{equation}\label{E:u_infty is a polynomial}
u_\infty|_T\in\PP_\ell(T).
\end{equation}
Theorem~\ref{T:eigenfunction} claims that $u_\infty$ is an eigenfunction of~(\ref{E:cont-problem}) and thus~(\ref{E:u_infty is a polynomial}) contradicts Assumption~\ref{A:non-deg}.
\end{proof}

\begin{remark}
 It is important to notice that the convergence of $h_k$ to zero is not an assumption, but a consequence of the fact that a subsequence is converging to an eigenfunction $u_\infty$ and the Non-Degeneracy Assumption~\ref{A:non-deg}.
\end{remark}

\begin{proof}[Proof of Theorem~\ref{T:main-result}]
In view of Theorem~\ref{T:cuasi main result} it remains to prove that $\lambda_k$ converges to the $j$-th eigenvalue of~(\ref{E:cont-problem}). By Lemma~\ref{L:nondeg->h0} the result follows from~(\ref{E:lambdak converges}).
\end{proof}

We conclude the article with several remarks.

\begin{remark}
At first sight, the convergence of $\|h_k\|_{L^\infty(\Omega)}$ to zero looks like a very strong statement, especially in the context of adaptivity. But the uniform convergence of the meshsize to zero should not be confused with quasi-uniformity of the sequence of triangulations $\{\Tau_k\}_{k\in\NN_0}$, the latter is not necessary for the former to hold. Thinking about this more carefully, we realize that if we wish to have (optimal) convergence of finite element functions to some given function in $H^1(\Omega)$, then $h_k$ \emph{must tend to zero everywhere} (pointwise) unless the objective function is itself a polynomial of degree $\le \ell$ in an open region of $\Omega$. Lemma~\ref{L:hktendstozero} implies that the convergence of $h_k$ to zero is also uniform, and this does not necessarily destroy optimality~\cite{CKNS-quasi-opt,Stevenson,Garau-Morin-Zuppa-quasi-eig}.
\end{remark}
 
\begin{remark}
A sufficient condition to guarantee that we converge to the desired eigenvalue is to assume that $h_k \to 0$ as $k\to\infty$. This condition is weaker than the Non-Degeneracy Assumption, but it is in general impossible to prove a priori.
\end{remark}

\begin{remark}
Another option to guarantee convergence to the desired eigenvalue is to start with a mesh which is sufficiently fine. In view of the minimum-maximum principles, it is sufficient to start with a triangulation $\Tau_0$ that is sufficiently fine to guarantee that $\lambda_{j,\Tau_0} < \lambda_{j_0}$, where $j_0 > j$ is the minimum index such that $\lambda_{j_0} > \lambda_{j}$. This condition is verifiable a posteriori if we have a method to compute eigenvalues approximating from below. Some ideas in this direction are presented in~\cite{Armentano-Duran}, where the effect of mass lumping on the computation of discrete eigenvalues is studied.
\end{remark}

\bibliographystyle{amsalpha}
\bibliography{bibliography}

{\setlength{\parindent}{0pt}
\small

\medskip
\rule{.5\textwidth}{1.5pt}

\medskip
\textbf{Affiliations}
\begin{description}
\item[Eduardo Garau:] Consejo Nacional de Investigaciones Cient\'{\i}ficas y T\'{e}cnicas and Universidad Nacional del Litoral, Argentina.\\
e-mail: \url{egarau@santafe-conicet.gov.ar}\\
Address: IMAL, G\"uemes 3450, S3000GLN Santa Fe, Argentina.

\item[Pedro Morin:] Consejo Nacional de Investigaciones Cient\'{\i}ficas y T\'{e}cnicas and Universidad Nacional del Litoral, Argentina.\\
e-mail: \url{pmorin@santafe-conicet.gov.ar}\\
Address: IMAL, G\"uemes 3450, S3000GLN Santa Fe, Argentina.

\item[Carlos Zuppa:] Universidad Nacional de San Luis, Argentina.\\
e-mail: \url{zuppa@unsl.edu.ar}\\
Address: Departamento de Matem\'atica, Facultad de Ciencias F\'\i sico, Matem\'aticas y Naturales, Universidad Nacional de San Luis, Chacabuco 918, 5700 San Luis, Argentina.
\end{description}
}

\end{document}